\documentclass{amsart}
\usepackage{amsmath}
\usepackage{amsfonts}
\usepackage{color}
\usepackage{verbatim}
\usepackage{eepic}
\usepackage{epic}
\usepackage{epsfig,subfigure}
\usepackage[]{graphics}

\def\tribar{\vert\thickspace\!\!\vert\thickspace\!\!\vert}

\def\Forall{\quad \hbox{for all }}

\begin{document}

\newcommand{\al}{\alpha}
\newcommand{\be}{\beta}
\newcommand{\om}{\omega}
\newcommand{\ga}{\gamma}
\newcommand{\wh}{\widehat}
\newcommand{\Ga}{\Gamma}
\newcommand{\ka}{\kappa}
\newcommand{\teps}{\tilde e}
\newcommand{\fy}{\varphi}
\newcommand{\Om}{\Omega}
\newcommand{\si}{\sigma}
\newcommand{\Si}{\Sigma}
\newcommand{\de}{\delta}
\newcommand{\De}{\Delta}
\newcommand{\la}{\lambda}
\newcommand{\La}{\Lambda}
\newcommand{\ep}{\epsilon}

\newcommand{\vth}{\vartheta}
\newcommand{\vtht}{{\widetilde \vartheta}}
\newcommand{\rh}{\varrho}
\newcommand{\rlh}{{\widetilde \varrho}}

\newcommand{\on}{\quad\text{on}\ }
\newcommand{\orr}{\quad\text{or}\ }
\newcommand{\inn}{\quad\text{in}\ }
\newcommand{\as}{\quad\text{as}\ }
\newcommand{\at}{\quad\text{at}\ }
\newcommand{\ifff}{\quad\text{if}\ }
\newcommand{\andy}{\quad\text{and}\ }
\newcommand{\with}{\quad\text{with}\ }
\newcommand{\when}{\quad\text{when}\ }
\newcommand{\where}{\quad\text{where}\ }
\newcommand{\FEM}{\text{finite element method }}
\newcommand{\FE}{\text{finite element }}
\def\tribar{\vert\thickspace\!\!\vert\thickspace\!\!\vert}

\def\Mitaga1{{E_{\al,1}}}
\def\Mitagaa{{E_{\al,\al}}}
\def\C{{\mathbb{C}}}
\def\V{{H^1_0(\Omega)}}

\def\NN{{N}}
\def\Etilh{{\bar{E}_h}}

%
\def\guh{{\underline u}_h}    
\def\guht{{\underline u}_{h,t}}    
\def\puht{u_{h,t}}                 
\def\puh{\widetilde{u}_h}                 
\def\luh{{\bar u}_h}     
\def\luht{{\bar u}_{h,t}}
\def\luhtt{{\bar u}_{h,tt}}
\def\p{q}
\def\q{\luht}
\def\luhp{\luht}
\def\L2o{{L_2(\Om)}}
\def\K{\tau}
\def\zK{z^\K}
\def\T{{\mathcal{T}}}
\def\theto{{\overline \theta}}
\def\thetu{{\underline \theta}}
\def\ue{{\underline e}}
\def\oe{{\overline e}}
\def\ee{e}
\def\Lah{{\Lambda}_h}
\def\psiu{{\psi_t}}
\def\ww{w}
\newcommand{\EE}{\tilde e}

\def\Dal{{\partial^\alpha_t}}
\def\Dt{{\partial_t}}
\def\Dalpl{{\partial^{\alpha+l}_t}}
\def\gamal{{\gamma}_\alpha}

\def\bDelh{{\bar{\Delta}_h}}
\def\Ftilh{\bar{E}_h}
\def\Pbarh{\bar{P}_h}
\def\Ebar{{F}}
\def\Ellipt{{\mathcal L}}


\numberwithin{equation}{section}
\newtheorem{theorem}{Theorem}[section]
\newtheorem{lemma}{Lemma}[section]
\newtheorem{corollary}{Corollary}[section]
\newtheorem{remark}[theorem]{Remark}
\newtheorem{proposition}{Proposition}[section]
\newtheorem{definition}{Definition}[section]

\newtheorem{assumption}{Assumption}
\renewcommand{\theassumption}{\Alph{assumption}}

\graphicspath{{../paper_weakdata/}}

\title[Galerkin FEM for fractional PDE's with non-smooth data]
{Galerkin FEM for fractional order parabolic \\
equations with initial data in $H^{-s},~0 < s \le 1$}

\author {Bangti Jin \and Raytcho Lazarov \and Joseph Pasciak \and Zhi Zhou} 

\address{
Mathematics, Texas A\&M University,
College Station, TX 77843, USA}
\thanks{The research of R. Lazarov and Z. Zhou 
was supported in parts by US NSF Grant
DMS-1016525 and J. Pasciak has been supported by NSF Grant 
DMS-1216551. The work of all authors has been  supported also  by Award No.
KUS-C1-016-04, made by King Abdullah University of
Science and Technology (KAUST)}



\maketitle

\begin{abstract}
We investigate semi-discrete numerical schemes based on the standard Galerkin
and lumped mass Galerkin finite element methods
for an initial-boundary value problem for homogeneous fractional
diffusion problems with non-smooth initial data.
We assume that $\Omega\subset \mathbb{R}^d$, $d=1,2,3$ is a convex polygonal (polyhedral) domain.
We theoretically justify optimal order error estimates in $L_2$- and $H^1$-norms
for initial data in $H^{-s}(\Omega),~0\le s \le 1$.
We confirm our theoretical findings with a number of numerical tests
that include initial data $v$ being a Dirac $\delta$-function supported on a $(d-1)$-dimensional manifold.

\end{abstract}

\section{Introduction}\label{sec:intro}

We consider the initial--boundary value problem for the fractional order parabolic
differential equation for $u(x,t)$:
\begin{alignat}{3}\label{eq1}
   \Dal u(x,t) + \Ellipt u(x,t)&= f(x,t),&&\quad x\text{ in  } \Omega&&\quad T \ge t > 0,\notag\\
   u(x,t)&=0,&&\quad x\text{ in}\  \partial\Omega&&\quad T \ge t > 0,\\
    u(x,0)&=v(x),&&\quad x\text{ in  }\Omega,&&\notag
\end{alignat}
where $\Omega\subset\mathbb{R}^d\, (d=1,2,3)$ is a bounded convex polygonal domain with a boundary
$\partial\Omega$, and $\Ellipt$ is a symmetric,
uniformly elliptic second-order differential
operator.  Integrating the second order derivatives by parts (once)
gives rise to a bilinear form $a(\cdot,\cdot)$ satisfying
$$a(v,w) = (\Ellipt v,w) \Forall v\in H^2(\Omega),w\in H^1_0(\Omega),$$
where $(\cdot,\cdot)$ denotes the inner product in $L_2(\Omega)$.
The form $a(\cdot,\cdot)$ extends continuously to
$H^1_0(\Omega) \times H^1_0(\Omega)$ where it is symmetric and coercive
and we take $\|u\|_{H^1} = a(u,u)^{1/2}$, for all $u\in H^1_0(\Omega)$.
Similarly, $\Ellipt $ extends continuously to an operator
from $ H^1_0(\Omega) $ to $H^{-1}(\Omega)$ (the set of bounded linear functionals on
$H^1_0(\Omega)$) by
\begin{equation}\label{form-a}
\langle \Ellipt u, v\rangle = a(u, v) \Forall u,v\in H^1_0(\Omega).
\end{equation}
Here $\langle\cdot,\cdot\rangle$ denotes duality pairing between
$H^{-1}(\Omega)$ and $ H^1_0(\Omega) $.   We assume that the
coefficients of $\Ellipt$ are smooth enough so that solutions $v\in
H^1_0(\Omega) $ satisfying
$$a(v,\phi)= (f,\phi)\Forall \phi \in H^1_0(\Omega)$$
with $f\in L_2(\Omega)$ are in $H^2(\Omega)$.

Here  $\Dal u$ ($0<\al<1$) denotes the left-sided Caputo fractional derivative of order
$\al$ with respect to $t$ and it is defined by
(cf. \cite[p.\,91]{KilbasSrivastavaTrujillo:2006}
or \cite[p.\,78]{Podlubny_book})
\begin{equation*}
   \Dal v(t)= \frac{1}{\Gamma(1-\al)} \int_0^t(t-\tau)^{-\al}\frac{d}{d\tau}v(\tau)\, d\tau,
\end{equation*}
where $\Gamma(\cdot)$ is the Gamma function. Note that as the fractional order $\al$ tends to unity,
the fractional derivative $\Dal u$ converges to the canonical first-order derivative
$\frac{du}{dt}$ \cite{KilbasSrivastavaTrujillo:2006}, and thus
\eqref{eq1} reproduces the standard parabolic equation.
The model \eqref{eq1} captures
well the dynamics of subdiffusion processes in which the mean
square variance grows slower than that in a Gaussian process \cite{BouGeo} and has
found a number of practical applications.
A comprehensive survey on fractional order differential equations arising
in viscoelasticity, dynamical systems in control theory,
electrical circuits with fractance, generalized voltage divider,
fractional-order multipoles in electromagnetism, electrochemistry,
 and model of neurons is provided in \cite{Debnath_2003}; see also \cite{Podlubny_book}.

The goal of this study is to develop, justify, and test a numerical technique for
solving \eqref{eq1} with non-smooth initial data $v\in H^{-s}(\Omega)$, $0\leq s\leq1$, a important case 
in various applications and typical in related inverse problems; see e.g.,
\cite{ChengNakagawaYamamoto:2009}, \cite[Problem (4.12)]{Sakamoto_2011} and
\cite{JinLu:2012,KeungZou:1998}. This includes the case of
$v$ being a delta-function supported on a $(d-1)$--dimensional manifold in $\mathbb R^d$, is
particularly interesting from both theoretical and practical points of view.

The weak form for problem \eqref{eq1} reads: find $ u (t)\in \V$ such that
\begin{equation}\label{weak}
\begin{split}
{( \Dal u,\chi)}+ a(u,\chi)&= {(f, \chi)},
\quad \forall \chi\in \V,\ T\ge t>0, \quad
u(0)=v.
\end{split}
\end{equation}
The folowing two results are known, cf. \cite{Sakamoto_2011}: 
(1) for $v \in L_2(\Omega)$ the problem \eqref{eq1} has a unique solution in $C([0,T];L_2(\Omega) \cap C((0,T];H^2(\Omega)\cap
H^1_0(\Omega))$ \cite[Theorem 2.1]{Sakamoto_2011};
(2) for $f\in L_\infty(0,T;L_2(\Omega))$,
problem \eqref{eq1} has a unique solution in $L_2(0,T;H^2(\Omega)\cap
H^1_0(\Omega))$ \cite[Theorem 2.2]{Sakamoto_2011}.

To introduce the semidiscrete FEM for problem \eqref{eq1} we follow standard notation in \cite{Thomee97}.
Let ${\{\T_h\}}_{0<h<1}$ be a family of regular partitions of the domain
$\Omega$ into $d$-simplexes, called finite elements,
with $h$ denoting the maximum diameter. Throughout, we assume that the triangulation $\T_h$ is
quasi-uniform, i.e., the diameter of the inscribed disk in the finite element
$\tau \in \T_h$ is bounded from below by $h$, uniformly on $\T_h$. The approximation
 $u_h$ will be sought in the finite element space $X_h\equiv
X_h(\Om)$ of continuous piecewise linear functions over $\T_h $:
\begin{equation*}
  X_h =\left\{\chi\in H^1_0(\Om): \ \chi ~~\mbox{is a linear function over}  ~~\K,  
 \,\,\,\,\forall \K \in \T_h\right\}.
\end{equation*}

The semidiscrete Galerkin FEM for problem \eqref{eq1}
is: find $ u_h (t)\in X_h$ such that
\begin{equation}\label{fem}
\begin{split}
 {( \Dal u_{h},\chi)}+ a(u_h,\chi)&= {(f, \chi)},
\quad \forall \chi\in X_h,\ T \ge t >0, \quad
u_h(0)=v_h,
\end{split}
\end{equation}
where $v_h \in X_h$ is an approximation of $v$. The choice of $v_h$
will depend on the smoothness of $v$. For smooth data, $v \in H^2(\Omega)\cap H_0^1(\Omega)$,
we can choose $v_h$ to be either the finite element interpolant or the Ritz projection $R_h v$
onto $X_h$. In the case of non-smooth data, $v\in L_2(\Omega)$, following Thom\'ee
\cite{Thomee97}, we shall take $v_h=P_hv$,
where $P_h$ is the  $L_2$-orthogonal
projection operator $P_h: L_2(\Omega) \to X_h$,
defined by $(P_h \phi,\chi) = (\phi,\chi)$, $\chi \in X_h$.
In the intermediate case, $v \in H^1_0(\Omega)$,  we can choose either $v_h=P_hv$ or $v_h=R_hv$.
The goal of this paper is to study the convergence rates of the semidiscrete Galerkin method \eqref{fem}
for initial data $v\in H^{-s}(\Omega)$, $0 \le s \le 1$ when $f=0$.

The rest of the paper is organized as follows. In Section \ref{sec:prelim} we briefly review
the regularity theory for problem \eqref{eq1}.
In Section \ref{sec:Galerkin} we motivate our study by considering a 1-D example with
initial data being a $\delta$--function. Then in Theorem \ref{Galerkin} we prove the main result:
for $0 \le s \le 1$, the following error bound holds
$$
\|u(t) -u_h(t)\| + h \|\nabla (u(t) -u_h(t))\|   \le C h^{2-s} t^{-\al}\ell_h \|v\|_{-s}, \quad \ell_h =|\ln h| .
$$
Further, in Section \ref{sec:Lumpedmass} we show a similar result for the
lumped mass Galerkin method. 
Finally, in Section \ref{sec:numerics} we present numerical results for test problems
with smooth, intermediate, non-smooth initial data and initial data that
is a $\delta$--function, all confirming our theoretical findings.


\section{Preliminaries}\label{sec:prelim}


For the existence and regularity of the solution to \eqref{eq1},
we need some notation and preliminary results.
For $s\ge -1$, we denote by $\dot H^s(\Om) \subset
H^{-1}(\Om)$ the Hilbert space induced by the norm
\begin{equation}\label{norm-s}
|v|_s^2=\sum_{j=1}^\infty\la_j^s\langle v,\fy_j\rangle^2 
\end{equation}
with
$\{\la_j\}_{j=1}^\infty$ and $\{\fy_j\}_{j=1}^\infty$ being respectively the Dirichlet eigenvalues and
the $L_2$-orthonormal eigenfunctions of $\Ellipt$.  As usual, we
identify functions $f$ in $L_2(\Omega)$ with the functional
$F$ in $H^{-1}(\Omega)$ defined by $\langle F,\phi\rangle = (f,\phi)$,
for all $\phi\in H^1_0(\Omega)$.
The set $\{\fy_j\}_{j=1}^\infty$, respectively, $\{\la_j^\frac{1}{2} \fy_j\}_{j=1}^\infty$,
forms an orthonormal basis in $L_2(\Omega)$, respectively,
$H^{-1}(\Omega)$. Thus $|v|_0=\|v\|=(v,v)^\frac{1}{2}$ is the norm
in $L_2(\Omega)$ and $|v|_{-1}= \|v\|_{H^{-1}(\Omega)}$ is the norm in
$H^{-1}(\Omega)$. It is easy to check that
$|v|_1=a(v,v)^{\frac{1}{2}}$ is also the norm in $H_0^1(\Omega)$. Note that
$\{\dot H^s(\Omega)\}$, $s\ge -1$ form a Hilbert scale of interpolation spaces.
Motivated by this, we denote $\|\cdot\|_{H^s}$ to be the norm on the
interpolation scale between $H^1_0(\Omega)$ and $L_2(\Omega)$ when $s$
is in $[0,1]$ and $\|\cdot\|_{H^{s}}$ to be the norm on the
interpolation scale between  $L_2(\Omega)$ and $H^{-1}(\Omega)$ when $s$
is in $[-1,0]$.  Thus, $\| \cdot \|_{H^s}$ and $|\cdot|_s$ provide equivalent
norms for $s\in [-1,1]$.

We further assume that the coefficients of the elliptic operator $\Ellipt$ are sufficiently smooth
and the polygonal domain $\Omega$ is convex, so that $|v|_2=\|\Ellipt v\|$ is
equivalent to the norm in $H^2(\Om)\cap H^1_0(\Om)$ (cf. the proof of
Lemma 3.1 of \cite{Thomee97}).

Now we introduce the operator $E(t)$ by
\begin{equation}\label{E-oper}
 E(t)v=\sum_{j=1}^\infty \Mitaga1(-\la_j t^\al) \, (v, \fy_j) \, \fy_j,
\mbox{   where    }
E_{\al,\beta}(z)= \sum_{k=0}^\infty \frac{z^k}{\Gamma(k\alpha+\beta)}.
\end{equation}
Here $E_{\al,\beta}(z)$
is the Mittag-Leffler function defined for $ z\in\mathbb{C}$ \cite{KilbasSrivastavaTrujillo:2006}.
The operator $E(t)$ gives a representation of the solution $u$ of \eqref{eq1} with a homogeneous
right hand side, so that for $f(x,t) \equiv 0$
we have $u(t)=E(t)v$. This representation follows from eigenfunction expansion
\cite{Sakamoto_2011}. Further, we introduce the operator $\bar{E}(t)$ defined for $\chi \in L_2(\Omega)$ as
\begin{equation}\label{Duhamel}
{\bar E}(t) \chi = \sum_{j=0}^\infty t^{\al-1} \Mitagaa(-\la_j t^\al)\,(\chi,\fy_j)\, \fy_j.
\end{equation}
The operators $E(t) $ and $ {\bar E}(t) $ together give the following representation of the solution of \eqref{eq1}:
\begin{equation}\label{represent}
u(t)=E(t)v + \int_0^t  {\bar E}(t-s) f(s) ds.
\end{equation}

Motivated by \cite{ChengNakagawaYamamoto:2009,Sakamoto_2011}, we will study the convergence of semidiscrete Galerkin methods for problem \eqref{eq1}
with very weak initial data, i.e., $v \in H^{-s}(\Omega)$, $0\le s \le 1$. Then the
following question arises naturally: in what sense should we understand the solution for such weak data?
Obviously, for any $t>0$ the function $u(t)=E(t)v$ satisfies equation \eqref{eq1}.
Moreover, by dominated convergence we have 
\begin{equation*}
\lim_{t \to 0+} | E(t)v -v |_{-s} =
\Big (\lim_{t \to 0+} \sum_{j=1}^\infty (\Mitaga1(-\la_j t^\al) -1)^2 \la_j^{-s} (v,\fy_j)^2 \Big )^\frac12=0
\end{equation*}
provided that $v \in H^{-s}(\Omega)$. Here $(v,\fy_j)= \langle v, \fy_j \rangle_{H^{-s},H^s} $
is well defined since $\fy_j \in H^1_0(\Omega)$.
Therefore, the function $u(t)=E(t) v  $
satisfies \eqref{eq1} and for $t \to 0$ it converges to $v$ in $H^{-s}$--norm. That is,
it is a weak solution to \eqref{eq1}; see also \cite[Proposition 2.1]{ChengNakagawaYamamoto:2009}.

For the solution of the homogeneous equation \eqref{eq1}, which is the object of our
study, we have the following stability and smoothing estimates.
\begin{theorem}\label{thm:fdereg}
Let $u(t)=E(t)v$ be the solution to problem \eqref{eq1} with $f\equiv 0$. Then for $t>0$ we have the
the following estimates:
\begin{itemize}
  \item[(a)] for $\ell=0$, $ 0 \le q \le p \le 2$ and for $\ell=1$, $0 \le p \le q \le 2$ and $q \le p+2$:
      \begin{equation}\label{smooth-h-2}
         |(\Dal)^\ell u(t) |_p \le Ct^{-\al(\ell+\frac{p-q}{2})}|v|_q,
     \end{equation}
  \item[(b)] for $  0 \le s \le  1$ and $0\le p+s \le 2 $
    \begin{equation}
         | \Dal u(t) |_{-s} \le Ct^{-\al}|v|_{-s}, \quad \mbox{and} \quad
         | u(t) |_p \le Ct^{-\frac{p+s}{2}\al}|v|_{-s}.
     \end{equation}
\end{itemize}
\end{theorem}
\begin{proof}
Part (a) can be found in \cite[Theorem 2.1]{Sakamoto_2011} and
\cite[Theorem 2.1]{Bangti_LZ_2012}. Hence we only show part (b). Note that for $t>0$,
\begin{equation*}
\begin{split}
  | u(t) |_p^2 &\le \sum_{j=0}^{\infty} \la_j^p |\Mitaga1(-\la_j t^\al)|^2
               |(v,\phi_j)|^2
               \le C\sum_{j=0}^{\infty}  \frac{\la_j^p}{(1+\la_j t^\al)^2}
               |(v,\phi_j)|^2 \\
               &\le C t^{-(p+s)\al} \sum_{j=0}^{\infty}\frac{(\la_j t^\al)^{p+s}}{(1+\la_j t^\al)^2} \la_j^s
               |(v,\phi_j)|^2 \\
               &\le C t^{-(p+s)\al} \sum_{j=0}^{\infty} \la_j^s
               |(v,\phi_j)|^2 = Ct^{-(p+s)\al}|v|_{-s}^2,  \\
\end{split}
\end{equation*}
which proves the second inequality of case (b).  The first estimate follows similarly
by noticing the identity $\Dal E_{\alpha,1}(-\lambda t^\alpha) =
-\lambda E_{\alpha,1}(-\lambda t^\alpha)$ \cite{KilbasSrivastavaTrujillo:2006}.
\end{proof}

We shall need some properties of the $L_2$-projection $P_h$ onto $X_h$.
\begin{lemma}\label{lem:Ph-0stable}
Assume that the mesh is quasi--uniform.  Then for $s\in [0,1]$,
$$\begin{aligned}
\|(I-P_h)w \|_{H^s} &\le Ch^{2-s} \|w\|_{H^2}, \Forall w\in
H^2(\Omega)\cap H^1_0(\Omega),
\end{aligned}
$$
and
$$
\begin{aligned}
\|(I-P_h)w \|_{H^s} &\le Ch^{1-s} \|w\|_{H^1}, \Forall w\in
H^1_0(\Omega).
\end{aligned}
$$
In addition, $P_h$ is stable on $H^s(\Omega)$ for $s\in [-1,0]$.
\end{lemma}

\begin{proof}
Since the mesh is quasi-uniform,
the $L_2$--projection operator $P_h$ is stable in $H^1_0(\Omega)$
\cite{BrambleXu:1991}.  This immediately implies its stability in
$H^{-1}(\Omega)$.  Thus, stability on $H^{-s}(\Omega)$ follows from
this, the trivial stability of $P_h$ on $L_2(\Omega)$ and interpolation.

Let $I_h$ be the finite element interpolation operator and $C_h$
be the Clement or Scott-Zhang interpolation operator. 
It follows from the  stability of $P_h$ in $L_2(\Omega)$ and $H^1_0(\Omega)$
that
$$\begin{aligned}
\|(I-P_h)w \|_{L_2} &\le \|(I-I_h)w \|_{L_2}\le
 Ch^2 \|w\|_{H^2},\Forall w\in
H^2(\Omega)\cap H^1_0(\Omega),\\
\|(I-P_h)w \|_{H^1} &\le C\|(I-I_h)w \|_{H^1} \le Ch \|w\|_{H^2},\Forall w\in
H^2(\Omega)\cap H^1_0(\Omega),\\
\|(I-P_h)w \|_{L_2} &\le \|(I-C_h)w \|_{L_2} \le Ch \|w\|_{H^1},\Forall w\in
H^1_0(\Omega),\\
\|(I-P_h)w \|_{H^1} &\le C \|w\|_{H^1},\Forall w\in
H^1_0(\Omega).
\end{aligned}
$$
The inequalities of the lemma follow by interpolation.
\end{proof}

\begin{remark}  All the norms
appearing in Lemma \ref{lem:Ph-0stable} can be
  replaced by their corresponding  equivalent dotted norms.
\end{remark}

\section{Galerkin finite element method}\label{sec:Galerkin}

To motivate our study we shall first consider the 1-D case, i.e.,  $\Ellipt u=-u''$,
and take initial data the Dirac $\delta$-function at $x=\frac{1}{2}$,
$\langle \delta,v \rangle = v(\frac{1}{2})$.
It is well known that $H_0^{\frac{1}{2}+\epsilon}(0,1)$ embeds continuously
into $C_0(0,1)$, hence the $\delta$-function is a bounded linear functional on
the space $H_0^{\frac{1}{2}+\epsilon}(\Omega)$, i.e.,  $ \delta \in H^{-\frac{1}{2}-\epsilon}(\Omega)$.

In Tables \ref{tab:GalerkinDel-2} and \ref{tab:LumpedDel-1} we show the
error and the convergence rates for the
semidiscrete Galerkin FEM and semidiscrete lumped mass FEM (cf. Section 4)
for initial data $v$ being a Dirac $\delta$-function at $x=\frac{1}{2}$.
The results suggest an $O(h^\frac{1}{2})$ and $O(h^\frac{3}{2})$ convergence
rate for the $H^1$- and $L_2$-norm of the error, respectively.
Below we prove that up to a factor $|\ln h|$ for fixed $t>0$,
the convergence rate is of the order reported in Tables  \ref{tab:GalerkinDel-2}
and \ref{tab:LumpedDel-1}. In Table \ref{tab:GalerkinDel-1} we show the
results for the case that the $\delta$-function is supported at a grid point.
In this case the standard Galerkin method converges at the expected rate in $H^1$-norm, while
the convergence rate in the $L_2$-norm is $O(h^2)$.
This is attributed to the fact that in 1-D
the solution with the $\delta$-function as the initial data
is smooth from both sides of the support point and the finite element spaces
have good approximation property.
\begin{table}[!ht]
\caption{Standard FEM with initial data $\delta(\frac12)$ for $h=1/(2^k+1)$, $\al=0.5$.}
\label{tab:GalerkinDel-2}
\begin{center}
     \begin{tabular}{|c|c|c|c|c|c|c|c |c|}
     \hline
      time & $k$ & $3$ & $4$ & $5$ &$6$ & $7$ & ratio & rate\\     
     \hline
     $t=0.005$ & $L_2$-norm & 3.95e-2 & 1.59e-2 & 6.00e-3 & 2.19e-3 & 7.89e-4 &  $\approx 2.75$ & $O(h^\frac32)$ \\
     \cline{2-9}
     & $H^1$-norm           & 1.21e0 & 8.99e-1 & 6.52e-1 & 4.66e-1 & 3.33e-1 & $\approx 1.40$ & $O(h^\frac12)$ \\
     \hline
     $t=0.01$ & $L_2$-norm  & 2.85e-2 & 1.13e-2 & 4.26e-3 & 1.55e-3 & 5.58e-4 & $\approx 2.77$  & $O(h^\frac32)$ \\
     \cline{2-9}
     & $H^1$-norm           & 8.66e-1 & 6.39e-1 & 4.62e-1 & 3.31e-1 & 2.35e-1 & $\approx 1.40$  & $O(h^\frac12)$ \\
     \hline
     $t=1$ & $L_2$-norm     & 3.04e-3 & 1.17e-3 & 4.34e-4 & 1.57e-4 & 5.61e-5 & $\approx 2.79$ & $O(h^\frac32)$ \\
     \cline{2-9}
     & $H^1$-norm           & 8.91e-2 & 6.49e-2 & 4.66e-2 & 3.32e-2 & 2.36e-2 & $\approx 1.41$ & $O(h^\frac12)$ \\
     \hline
     \end{tabular}
\end{center}
\end{table}
\begin{table}[!ht]
\caption{Lumped mass FEM with initial data $\delta(\frac12)$, $h=1/2^k$ $\al=0.5$.}
\label{tab:LumpedDel-1}
\begin{center}
     \begin{tabular}{|c|c|c|c|c|c|c|c |c|}
     \hline
      time & $k$ & $3$ & $4$ & $5$ &$6$ & $7$ & ratio & rate \\     
     \hline
     $t=0.005$ & $L_2$-norm & 7.24e-2 & 2.66e-2 & 9.54e-3 & 3.40e-3 & 1.21e-3 & $\approx 2.79$ & $O(h^\frac32)$ \\
     \cline{2-9}
     & $H^1$-norm           & 1.51e0 & 1.07e0 & 7.60e-1 & 5.40e-1 & 3.81e-1 & $\approx 1.41$  & $O(h^\frac12)$ \\ 
     \hline
     $t=0.01$ & $L_2$-norm  & 5.20e-2 & 1.89e-2 & 6.77e-3 & 2.40e-3 & 8.54e-4 & $\approx 2.79$ & $O(h^\frac32)$ \\ 
     \cline{2-9}
     & $H^1$-norm           & 1.07e0 & 7.59e-1 & 5.37e-1 & 3.80e-1 & 2.70e-1 & $\approx 1.41$ & $O(h^\frac12)$ \\
     \hline
     $t=1$ & $L_2$-norm     & 5.47e-3 & 1.93e-3 & 6.84e-4 & 2.42e-4 & 8.56e-5 & $\approx 2.79$ & $O(h^\frac32)$ \\ 
     \cline{2-9}
     & $H^1$-norm           & 1.07e-1 & 7.58e-2 & 5.37e-2 & 3.80e-2 & 2.70e-2 & $\approx 1.41$ & $O(h^\frac12)$ \\ 
     \hline
     \end{tabular}
\end{center}
\end{table}

\begin{table}[!ht]
\caption{Standard semidiscrete FEM with initial data $\delta(\frac12)$, $h=1/2^k$, $\al=0.5$.}
\label{tab:GalerkinDel-1}
\begin{center}
     \begin{tabular}{|c|c|c|c|c|c|c|c|c|}
     \hline
     Time& $k$& $3$ & $4$ &$5$ &$6$ & $7$ & ratio &  rate \\
     \hline
     $t=0.005$ & $L_2$-norm  & 5.13e-3 & 1.28e-3 & 3.21e-4 & 8.03e-5 & 2.01e-5 & $\approx 3.99$ & $O(h^2)$ \\
     \cline{2-9}
     & $H^1$-norm           & 4.29e-1 & 3.09e-1 & 2.21e-1 & 1.56e-1 & 1.11e-1 & $\approx 1.41$ & $O(h^\frac12)$ \\ 
     \hline
     $t=0.01$ & $L_2$-norm  & 3.07e-3 & 7.70e-4 & 1.93e-4 & 4.82e-5 & 1.21e-5 & $\approx 3.98$ & $O(h^2)$ \\ 
     \cline{2-9}
     & $H^1$-norm           & 3.04e-1 & 2.19e-1 & 1.56e-1 & 1.11e-1 & 7.87e-2 & $\approx 1.41$ & $O(h^\frac12)$ \\
     \hline
      $t=1$ & $L_2$-norm     & 1.44e-5 & 2.64e-6 & 6.66e-7 & 1.69e-7 & 4.30e-8 & $\approx 3.94$ & $O(h^2)$ \\
     \cline{2-9}
     & $H^1$-norm           &  3.15e-2 &  2.23e-2 &  1.58e-2 & 1.11e-2 & 7.81e-3 & $\approx 1.41$ & $O(h^\frac12)$ \\
     \hline
     \end{tabular}
\end{center}
\end{table}
The numerical results in Tables \ref{tab:GalerkinDel-2}--\ref{tab:GalerkinDel-1}
motivate our study on the convergence rates of the semidiscrete Galerkin and lumped mass
schemes for initial data $v \in H^{-s}(\Omega)$, $0\le s \le 1$.

\begin{theorem}\label{Galerkin}
Let $u$  and $u_h$ be the solutions of \eqref{eq1} and the semidiscrete
Galerkin finite element method \eqref{fem} with $v_h=P_hv$, respectively.
Then there is a constant $C>0$ such that for $0 \le s \le 1$
\begin{equation}\label{main-est}
 \| u_h(t) - u(t) \| + h  \|\nabla(u_h(t) - u(t))\|
   \le C h^{2-s} \, \ell_h \,t^{-\al}|v|_{-s}.
\end{equation}
\end{theorem}
\begin{remark}
 Note that for any fixed $\epsilon$ there is a $C_\epsilon >0$ such that
$|\delta |_{-\frac{1}{2}-\epsilon} \le C_\epsilon$. Thus, modulo the factor $\ell_h=|\ln h|$,
the theorem confirms the computational results of Table \ref{tab:GalerkinDel-2},
namely convergence in the $L_2$--norm with a rate $O(h^\frac32)$ and in
$H^1$--norm with a rate $O(h^\frac12)$.
\end{remark}
\begin{proof}

We shall need the following auxiliary problem: find $u^h(t) \in H^1_0(\Omega)$, s.t.
\begin{equation}\label{weak-ID}
\begin{split}
 {( \Dal u^h(t),\chi)}+ a(u^h(t),\chi)&= {(f(t), \chi)},
~~ \forall \chi\in \V, ~t >0,
~u^h(0)=P_h v.
\end{split}
\end{equation}
We note that the initial data $u^h(0)=P_hv\in H^1_0(\Omega)$ is smooth.

Now we consider the semidiscrete Galerkin method for problem \eqref{weak-ID},
i.e., equation \eqref{fem} with $v_h =P_hv$.
By Theorem 3.2 of \cite{Bangti_LZ_2012} we have
\begin{equation}
\begin{split}
 \| u_h(t) - u^h(t) \| + h  \|\nabla(u_h(t) - u^h(t))\|
 &  \le C h^{2} \, \ell_h \,t^{-\al} \|P_h v \|.
\end{split}
\end{equation}
Now, using the inverse inequality
$\|P_h v\| \le C h^{-s} \|P_h v\|_{-s}$, for $0\le s \le 1$,
and the stability of $P_h$ in $H^{-s}(\Omega)$ (cf. Lemma \ref{lem:Ph-0stable}),
we get
\begin{equation}\label{joe-est}
 \| u_h(t) - u^h(t) \| + h  \|\nabla(u_h(t) - u^h(t))\|
   \le C h^{2-s} \, \ell_h \,t^{-\al} \|v \|_{-s}.
\end{equation}

Now we estimate $u(t) - u^h(t)=E(t)(v - P_h v)$. To this end,
let $\{v_n\}\subset L_2(\Omega)$ be a sequence converging to $v$ in
$H^{-s}(\Omega)$. Noting that the operators $P_h$
and $E(t)$ are self-adjoint in $(\cdot,\cdot)$ and using
the smoothing property \eqref{smooth-h-2}  of $E(t)$
with $\ell =0$, $q=0$ and $p=2$,  we obtain for any $\phi\in L_2(\Omega)$
\begin{equation*}
\begin{aligned}
 | (E(t)(I- P_h)v_n,\phi)|&=  |( v_n,(I-P_h)E(t)\phi)|
                    \le |v_n|_{-s} |(I-P_h)E(t)\phi|_s\\
                    &\le C h^{2-s} |v_n|_{-s} |E(t)\phi|_2
                     \le C  h^{2-s}t^{-\al} |v_n|_{-s}\|\phi\|.
\end{aligned}
\end{equation*}
Taking the limit as $n$ tends to infinity gives
\begin{equation}\label{L2-error}
\begin{aligned}
\| u(t) - u^h(t) \| &= \sup_{\phi \in L_2(\Omega)} \frac{|(E(t)(I-P_h)v,\phi)|}{\| \phi \|} \le C  h^{2-s}t^{-\al} |v|_{-s}.
\end{aligned}
\end{equation}
Then by the triangle inequality we arrive at the $L_2$-estimate in \eqref{main-est}.

Next, for the gradient term $\|\nabla(u(t)-u^h(t))\|$, we observe that
for any $\phi\in \dot H^1(\Omega)$, by the coercivity of
$a(\cdot,\cdot)$, we have
\begin{equation}
\begin{aligned}
C_0 \|\nabla(E(t) (I-P_h) v_n) \|^2 &\le a( E(t) (I-P_h) v_n,E(t) (I-P_h) v_n)\\
&\le \sup_{\phi\in H^1_0(\Omega)}  \frac{ a( E(t) (I-P_h) v_n,\phi)^2}
{a(\phi,\phi)}.
\end{aligned}
\label{intermed}
\end{equation}
Meanwhile we have
$$\begin{aligned}
|a( E(t) (I-P_h) v_n,\phi)|&=  |((I-P_h) v_n,E(t)\Ellipt \phi)|
=| ( v_n,(I-P_h)E(t)\Ellipt \phi)|\\
&\le C|v_n|_{-s} |(I-P_h)E(t)\Ellipt \phi|_{s}
\le C h^{1-s} |v_n|_{-s} |E(t)\Ellipt \phi|_1\\
&\le C  h^{1-s} t^{-\al}|v_n|_{-s} |\Ellipt \phi|_{-1}
                    \le C  h^{1-s}t^{-\al} |v_n|_{-s} |\phi|_1.
\end{aligned}
$$
Passing to the limit as $n$ tends to infinity
and combining with \eqref{intermed} gives
\begin{equation}\label{H1-error}
\|\nabla  (u(t) - u^h(t))\|\le C  h^{1-s}t^{-\al} |v|_{-s}.
\end{equation}
Thus, \eqref{L2-error} and \eqref{H1-error} lead to the following estimate for $0\le s \le 1$:
\begin{equation}\label{auxiliary}
 \| u(t) - u^h(t) \| + h \|\nabla  (u(t) - u^h(t))\|
\le C  h^{2-s}t^{-\al} |v|_{-s}.
\end{equation}
Finally,
\eqref{joe-est}, \eqref{auxiliary}, and the triangle inequality give the desired estimate \eqref{main-est}
 and this completes the proof.
\end{proof}

\section{Lumped mass method}\label{sec:Lumpedmass}
%
In this section, we consider the lumped mass FEM in planar domains (see, e.g. \cite[Chapter 15, pp.
239--244]{Thomee97}).
An important feature of the lumped mass method is that when representing the solution
$\luh$ in the nodal basis functions,
the mass matrix is diagonal. This leads to a simplified computational procedure.
For completeness we shall briefly describe this approximation. Let $z^\tau_j
$, $j=1,\dots,d+1$ be the vertices of the $d$-simplex $\tau \in \T_h$.
Consider the following quadrature formula and the induced inner product in $X_h$:
\begin{equation*}
Q_{\K,h}(f) = \frac{|\tau|}{d+1} \sum_{j=1}^{d+1} f(z^\tau_j) \approx \int_\tau f dx, \quad
(w, \chi)_h = \sum_{\tau \in \T_h}  Q_{\tau,h}(w \chi)
\end{equation*}
Then lumped mass finite element method is: find $ \luh (t)\in X_h$ such that
\begin{equation}\label{fem-lumped}
\begin{split}
 {(\Dal \luh, \chi)_h}+ a(\luh,\chi)&= (f, \chi)
\quad \forall \chi\in X_h,\ t >0, ~~~
\luh(0)=P_h v.
\end{split}
\end{equation}

To analyze this scheme we shall need the concept of \textit{symmetric} meshes. Given a vertex
$z\in \T_h$, the patch $\Pi_z$ consists of all finite elements having $z$ as a vertex.
A mesh $\T_h$ is said to be symmetric at the vertex $z$, if
$x \in \Pi_z$ implies $2z - x \in \Pi_z$, and
$\T_h$ is symmetric if it is symmetric at every interior vertex.

In \cite[Theorem 4.2]{Bangti_LZ_2012} it was shown
that if the mesh
is symmetric, then the lumped mass scheme \eqref{fem-lumped} for $f=0$
has an almost optimal convergence rate in $L_2$-norm for nonsmooth data $v\in L_2(\Omega)$.

Now we prove the main result concerning the lumped mass method:
\begin{theorem}\label{est-lumped-mass}
Let $u(t)$ and $\luh(t)$ be the solutions of the problems \eqref{eq1} and \eqref{fem-lumped},
respectively. Then for $t > 0$ the following error estimate is valid:
\begin{equation}\label{lumped-H1}
\| \luh(t) -u(t) \| +
\|\nabla( \luh(t) -u(t) )\|\le C h^{1-s}\ell_h t^{-
\al}\|v \|_{-s}, ~~~0 \le s \le 1.
\end{equation}
Moreover, if the mesh is symmetric then
\begin{equation}\label{lumped-L2}
 \| \luh(t) -u(t) \| \le C h^{2-s}\ell_h t^{-\al}\|v \|_{-s}, ~~~0 \le s \le 1.
\end{equation}
\end{theorem}
\begin{proof}
We split the error into $\luh(t) -u(t) = \luh(t) -u^h(t) + u^h(t) -u(t)$, where $ u^h(t) -u(t)$
was estimated in \eqref{auxiliary}.
The term $ \luh(t) -u^h(t)$ is the error of the lumped mass method for the auxiliary
problem \eqref{weak-ID}. Since the initial data $P_h v \in L_2(\Om)$,
we can apply known estimates on $\luh(t)-u^h(t)$ \cite[Theorem 4.2]{Bangti_LZ_2012}. Namely,\\
(a) If the  mesh is globally quasiuniform, then
$$
\| \luh(t) -u^h(t) \| + h \| \nabla (\luh(t) -u^h(t))\| \le C h t^{-\al} \ell_h \|P_h v\|;
$$
(b) If the mesh is symmetric, then
$$
\| \luh(t) -u^h(t) \|  \le C h^{2} t^{-\al} \ell_h \|P_h v\|.
$$
These two estimates, the inequality
$\|P_h v\| \le Ch^{-s} \|v\|_{-s}$, $0 \le s \le 1$,
and estimate \eqref{joe-est}
give the desired result. This completes the proof of the theorem.
\end{proof}
\begin{remark}
The $H^1$-estimate is almost optimal for any quasi-uniform meshes, while
the $L_2$-estimate is almost optimal for symmetric meshes. For the
standard parabolic equation with initial data $v\in L_2(\Omega)$, it was shown
in \cite{chatzipa-l-thomee-FV} that the lumped mass scheme can achieve at most
an $O(h^\frac32)$ convergence order in $L_2$-norm for some nonsymmetric meshes.
This rate is expected to hold for fractional order differential equations as well.
\end{remark}

\section{Numerical results}\label{sec:numerics}

Here we present numerical results in 2-D to verify the error estimates derived herein and \cite{Bangti_LZ_2012}.
The 2-D problem \eqref{eq1} is on the unit square $\Omega=(0,1)^2$ with $\Ellipt =-\Delta$.
We perform numerical tests on four different examples:
\begin{enumerate}
 \item[(a)]  Smooth initial data: $v(x,y)= x(1-x)y(1-y)$; in this case the initial data $v$ is
  in $H^2(\Om)\cap H_0^1(\Om)$, and the exact solution $u(x,t)$ can be represented by a rapidly
  converging Fourier series:
  \begin{equation*}
  u(x,t)= \sum^\infty_{n=1}\sum^\infty_{m=1}\frac{4 c_{n}c_{m}}{m^3n^3\pi^6}E_{\al,1}(-\la_{n,m} t^\al) \sin (n\pi x) \sin (m \pi y),
  \end{equation*}
where $\la_{n,m}=(n^2+m^2)\pi^2$,  and
$c_{l}=4\sin^2(l\pi/2) - l \pi \sin(l\pi)$, $l=m,n$.

 \item[(b)]  Initial data in $ H^1_0(\Omega)$ (case of intermediate smoothness):
 \begin{equation*}
    v(x) = (x-\tfrac12)(x-1)(y-\tfrac12)(y-1)\chi_{[\frac12,1]\times[\frac12,1]},\\
 \end{equation*}
where $ \chi_{[\frac12,1]\times[\frac12,1]}$ is the characteristic function of
$ [\frac12,1]\times[\frac12,1]$.

 \item[(c)]  Nonsmooth initial data: $ v(x) = \chi_{[\frac14,\frac34]\times[\frac14,\frac34]}$.

\item[(d)] Very weak data: $v=\delta_\Gamma $ 
with $\Gamma$ being the boundary of the square $[\frac14,\frac34]\times[\frac14,\frac34]$ with
$\langle \delta_\Gamma,\phi\rangle = \int_\Gamma \phi(s) ds$. One may view $(v,\chi)$ for $\chi \in X_h \subset
\dot H^{\frac12+\epsilon}(\Om)$ as duality pairing between the spaces
$H^{-\frac12-\epsilon}(\Om)$ and $\dot H^{\frac12+\epsilon}(\Om)$ for any $\epsilon >0$
so that 
$\delta_\Gamma \in H^{-\frac12-\epsilon}(\Omega)$.
Indeed, it follows from H\"{o}lder's inequality
\begin{equation*}
  \begin{aligned}
   \|\delta_\Gamma\|_{H^{-\frac{1}{2}-\epsilon}(\Omega)}&= \sup_{\phi \in \dot H^{\frac12+\epsilon}(\Omega)}
   \frac{|\int_\Gamma \phi(s)ds|}{\|\phi\|_{\frac12+\epsilon,\Omega}}
   \le  |\Gamma|^\frac12 \sup_{\phi \in \dot H^{\frac12+\epsilon} (\Omega)}\frac{\|\phi\|_{L_2(\Gamma)}}{\|\phi\|_{\frac12+\epsilon,\Omega}},
 \end{aligned}
 \end{equation*}
and the continuity of the trace operator
from $\dot H^{\frac{1}{2}+\epsilon}(\Omega)$ to $L_2(\Gamma)$. 
\end{enumerate}

The exact solution for each example can be expressed by an infinite series involving
the Mittag-Leffler function $E_{\alpha,1}(z)$. To accurately evaluate the Mittag-Leffler functions,
we employ the algorithm developed in \cite{Seybold:2008}.
To discretize the problem, we divide the unit interval $(0,1)$ into
$N=2^k$ equally spaced subintervals, with a mesh size $h=1/N$ so that $[0,1]^2$ is divided into $N^2$
small squares. We get a symmetric mesh for the domain $[0,1]^2$ by connecting the
diagonal of each small square. All the meshes we have used are symmetric and therefore
both semidiscrete Galerkin FEM and lumped mass FEM have the same theoretical accuracy. Unless otherwise specified,
we have used the lumped mass method.

To compute a reference (replacement of the exact) solution we have used two different numerical techniques
on very fine meshes.
The first is based on the exact representation of the semidiscrete lumped mass
solution $\bar{u}_h$ by
$$
\luh(t)= \sum_{n, m=1}^{N-1} E_{\al,1} (-\la^h_{n,m}t^\al)(v,\fy^h_{n,m})  \fy^h_{n,m},
$$
where $\fy^h_{n,m}(x,y) = 2 \sin(n\pi x) \sin(m\pi y)$, $n,m=1, \dots, N-1$,
with $x,y$ being grid points, are the discrete eigenfunctions and
$$
\la^h_{n,m}= \frac{4}{h^2} \left (\sin^2\frac{n\pi h}{2} + \sin^2\frac{m\pi h}{2} \right )
$$
are the corresponding eigenvalues.

The second numerical technique is based on fully discrete scheme, i.e.,
discretizing the time interval $[0,T]$ into $t_n=n\tau$,
$n=0,1,\dots$, with $\tau$ being the time step size, and then approximating the fractional derivative
$\Dal u(x,t_n)$ by finite difference \cite{LinXu:2007}: 
\begin{equation}\label{FDiff}
  \begin{aligned}
 \Dal u(x,t_n)
  &\approx\frac{1}{\Gamma(2-\alpha)}\sum_{j=0}^{n-1}b_j\frac{u(x,t_{n-j})-u(x,t_{n-j-1})}{\tau^\alpha},
  \end{aligned}
\end{equation}
where the weights $b_j=(j+1)^{1-\alpha}-j^{1-\alpha}$, $j=0,1,\ldots,n-1$.
This fully discrete solution is denoted by $U_h$.
Throughout, we have set $\tau=10^{-6}$ so
that the error incurred by temporal discretization is negligible (see Table \ref{tab:nonsmooth2D12} 
for an illustration).

We measure the accuracy of the approximation $u_h(t)$
by the normalized error $\|u(t)-u_h(t)\|/\|v\|$ and $\|\nabla (u(t)-u_h(t))\|/ \|v\|$.
The normalization enables us to observe the behavior of the error with respect to time in
case of nonsmooth initial data.
\paragraph{Smooth initial data: example (a).}
In Table \ref{tab:smooth2D1} we show the numerical results for $t=0.1$ and $\al=0.1,~0.5,~0.9$.
Here \texttt{ratio} refers to the ratio
between the errors as the mesh size $h$ is halved. In Figure \ref{fig:smooth2D}, we plot
the results from Table \ref{tab:smooth2D1} in a log-log scale. The slopes of the error
curves are $2$ and $1$, respectively, for $L_2$- and $H^1$-norm of the error.
This confirms the theoretical result from \cite{Bangti_LZ_2012}.
%
\begin{table}[h!]
\caption{Numerical results for smooth initial data, example (a), $t=0.1$.}
\label{tab:smooth2D1}
\begin{center}
     \begin{tabular}{|c|c|c|c|c|c|c|c|c|}
     \hline
     $\alpha$ & $h$ & $1/8$ & $1/16$ &$1/32$ &$1/64$ & $1/128$ & ratio & rate \\
     \hline
     $0.1$ & $L_2$-norm  & 9.25e-4 & 2.44e-4 & 6.25e-5 & 1.56e-5 & 3.85e-6 & $\approx 4.01$ & $O(h^2)$ \\
     \cline{2-9}
     & $H^1$-norm  & 3.27e-2 & 1.66e-2 & 8.40e-3 & 4.21e-3 & 2.11e-3 & $\approx 1.99$ & $O(h)$ \\
    \hline
     $0.5$ & $L_2$-norm  & 1.45e-3 & 3.84e-4& 9.78e-5 & 2.41e-5& 5.93e-6& $\approx 4.02$ & $O(h^2)$ \\
     \cline{2-9}
     & $H^1$-norm  & 5.17e-2 & 2.64e-2 & 1.33e-2 & 6.67e-3 & 3.33e-3 & $\approx 1.99$ & $O(h)$ \\
     \hline
     $0.9$ & $L_2$-norm  & 1.88e-3 & 4.53e-4 & 1.13e-4 & 2.82e-5 & 7.06e-6& $\approx 4.00$ & $O(h^2)$ \\
     \cline{2-9}
     & $H^1$-norm  & 6.79e-2 & 3.43e-2 & 1.73e-2 & 8.63e-3 & 4.31e-3 & $\approx 2.00$ & $O(h)$ \\
     \hline
     \end{tabular}
\end{center}
\end{table}
\begin{figure}[h!]
\center
  \includegraphics[width=11cm]{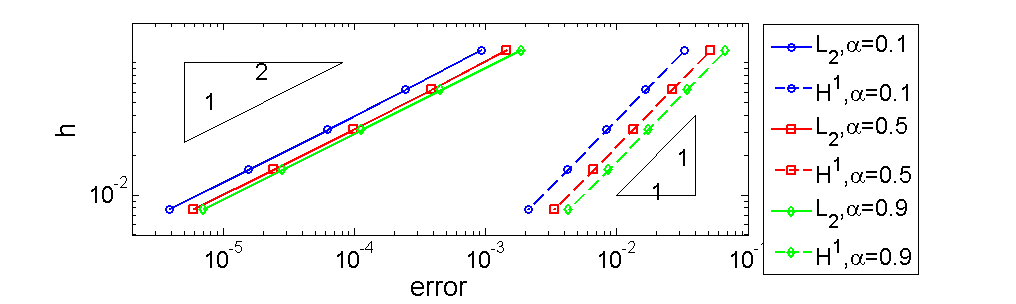}
  \caption{Error plots for smooth initial data, Example (a):
  $\al=0.1, 0.5, 0.9$ at $t=0.1$.}\label{fig:smooth2D}
\end{figure}
\paragraph{Intermediate smooth data: example (b).}
In this example the initial data $v(x)$ is in $H_0^1(\Om)$ and 
the numerical results are shown in Table \ref{tab:intermediate2D}
The slopes of the error curves in a log-log
scale are $2$ and $1$ respectively for $L_2$- and $H^1$-norm of the errors,
which agrees well with the theory for the intermediate case \cite{Bangti_LZ_2012}.
\begin{table}[h!]
\caption{Intermediate case (b) with $\al=0.5$ at $t=0.1$.}\label{tab:intermediate2D}
\begin{center}
     \begin{tabular}{|c|c|c|c|c|c|c|c|}
     \hline
     $h$ & $1/8$ & $1/16$ & $1/32$& $1/64$& $1/128$& ratio &  rate \\
     \hline
     $L_2$-error & 3.04e-3 & 8.20e-4 & 2.12e-4 & 5.35e-5 & 1.32e-5 & $\approx 3.97$ & $O(h^2)$ \\
     \hline
     $H^1$-error & 5.91e-2 & 3.09e-2 & 1.56e-2 & 7.88e-3 & 3.93e-3 & $\approx 1.98$ & $O(h)$\\
     \hline
     \end{tabular}
\end{center}
\end{table}
\paragraph{Nonsmooth initial data: example (c).}
First in Table \ref{tab:nonsmooth2D12} we compare
fully discrete solution $U_h$ via the finite difference approximation \eqref{FDiff} 
with the semidiscrete lumped mass solution $\bar{u}_h$ via eigenexpansion to study the error incurred by time discretization.
We observe that for each fixed spatial mesh size $h$, the difference between $\bar{u}_h$, the lumped mass FEM solution, and $U_h$
decreases with the decrease of $\tau$. In particular,
for time step $\tau=10^{-6}$ the error incurred by the time discretization is negligible, so the
fully discrete solutions $U_h$ could well be used as reference solutions.
\begin{table}[h!]
\caption{The difference $\bar{u}_h -U_h$, nonsmooth initial data, example (c): $\al=0.5$, $t=0.1$}
\label{tab:nonsmooth2D12}
\begin{center}
     \begin{tabular}{|c|c|c|c|c|c|c|}
     \hline
     Time step & $h$ & $1/8$ & $1/16$ &$1/32$ &$1/64$ & $1/128$  \\
     \hline
     $\tau=10^{-2}$ & $L_2$-norm  &2.03e-3  &2.01e-3  &2.00e-3  &2.00e-3  &2.00e-3  \\
     \cline{2-7}
     & $H^1$-norm  & 9.45e-3 & 9.17e-3  &9.10e-3  &9.08e-3  &9.07e-3 \\
    \hline
     $\tau=10^{-4}$ & $L_2$-norm  & 1.81e-5  &1.79e-5  &1.79e-5  &1.79e-5  &1.79e-5 \\
     \cline{2-7}
     & $H^1$-norm  & 8.47e-5 & 8.22e-5  &8.15e-5  &8.13e-5  &8.13e-5  \\
     \hline
     $\tau=10^{-6}$ & $L_2$-norm  & 1.80e-7  &1.78e-7  &1.78e-7  &1.78e-7  &1.78e-7\\
     \cline{2-7}
     & $H^1$-norm  & 8.42e-7  &8.17e-7  &8.10e-7 & 8.10e-7   &8.09e-7 \\
     \hline
     \end{tabular}
\end{center}
\end{table}
In Table \ref{tab:nonsmooth2D1} and Figure \ref{fig:nonsmooth2D}  we present the
numerical results for problem (c). 
These numerical results fully confirm the theoretically
predicted rates for nonsmooth initial data. 
%
\begin{figure}[h!]
\center
\includegraphics[width=11cm]{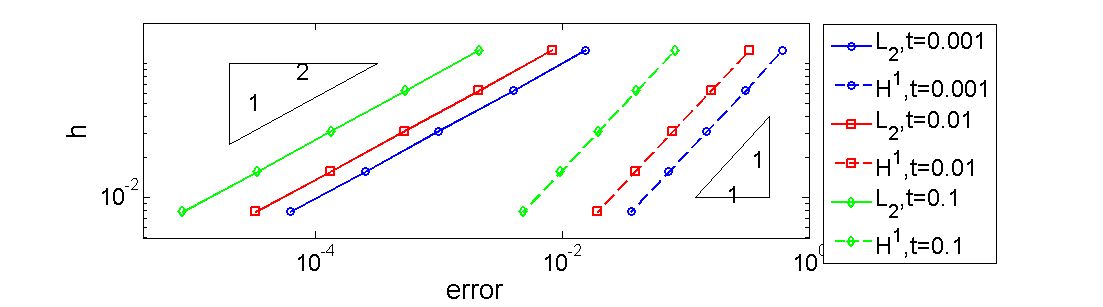}
\caption{Error plots for lumped FEM for nonsmooth initial data, Example (c):  $\al=0.5$.}
\label{fig:nonsmooth2D}
\end{figure}
\begin{table}[h!]
\caption{Error for the lumped FEM for nonsmooth initial data, example (c): $\al=0.5$} 
\label{tab:nonsmooth2D1}
\begin{center}
     \begin{tabular}{|c|c|c|c|c|c|c|c|c|}
     \hline
     Time& $h$ & $1/8$ & $1/16$ &$1/32$ &$1/64$ & $1/128$ & ratio &  rate \\
     \hline
     $t=0.001$ & $L_2$-norm  &1.55e-2  &3.99e-3  &1.00e-3  &2.52e-4  &6.26e-5 & $\approx 4.01$ &$O(h^2)$ \\ 
     \cline{2-9}
     & $H^1$-norm  & 6.05e-1 & 3.05e-1  &1.48e-1  &7.29e-2  &3.61e-2 & $\approx 2.00$ & $O(h)$ \\ 
    \hline
     $t=0.01$ & $L_2$-norm  & 8.27e-3  &2.10e-3  &5.28e-4  &1.32e-4  &3.29e-5& $\approx 4.01$ &$O(h^2)$ \\ 
     \cline{2-9}
     & $H^1$-norm  & 3.32e-1 & 1.61e-1  &7.90e-2  &3.90e-2  &1.93e-2 & $\approx 2.02$ & $O(h)$ \\ 
     \hline
     $t=0.1$ & $L_2$-norm  & 2.12e-3  &5.36e-4  &1.34e-4  &3.36e-5  &8.43e-6& $\approx 3.99$ &$O(h^2)$ \\ 
     \cline{2-9}
     & $H^1$-norm  & 8.23e-2  &4.01e-2  &1.96e-2 & 9.72e-3  &4.84e-3 & $\approx 2.01$ & $O(h)$ \\ 
     \hline
     \end{tabular}
\end{center}
\end{table}

\paragraph{Very weak data: example (d).}
The empirical convergence rate for the weak data $\delta_\Gamma$ agrees well with the theoretically
predicted convergence rate in Theorem \ref{Galerkin}, which gives
a ratio of $2.82$ and $1.41$, respectively, for the $L_2$- and $H^1$-norm of the error;
see Table \ref{tab:Deltafun}. Interestingly, for the standard Galerkin scheme, the
$L_2$-norm of the error exhibits super-convergence; see Table \ref{tab:Galerkinweak}.

\begin{table}[!ht]
\caption{Error for standard FEM: initial data Dirac $\delta$-function, $\al=0.5$}
\label{tab:Galerkinweak}
\begin{center}
     \begin{tabular}{|c|c|c|c|c|c|c|c|c|}
     \hline
     Time& $h$& $1/8$ & $1/16$ &$1/32$ &$1/64$ & $1/128$ & ratio &  rate\\
     \hline
     $t=0.001$ & $L_2$-norm  & 5.37e-2 & 1.56e-2 & 4.40e-3 & 1.23e-3 & 3.41e-4 & $\approx 3.57$ & $O(h^{1.84})$ \\
     \cline{2-9}
     & $H^1$-norm           & 2.68e0 & 1.76e0 & 1.20e0 & 8.21e-1 & 5.68e-1 &$\approx 1.45$ & $O(h^\frac12)$\\
     \hline
     $t=0.01$ & $L_2$-norm  & 2.26e-2 & 6.20e-3 & 1.67e-3 & 4.46e-4 & 1.19e-4 & $\approx 3.74$ &  $O(h^{1.90})$ \\
     \cline{2-9}
     & $H^1$-norm           & 9.36e-1 & 5.90e-1 & 3.92e-1 & 2.65e-1 & 1.84e-1 & $\approx 1.46$ &  $O(h^\frac12)$\\
     \hline
     $t=0.1$ & $L_2$-norm   & 8.33e-3 & 2.23e-3 & 5.90e-3 & 1.55e-3 & 4.10e-4 & $\approx 3.77 $ & $O(h^{1.91})$ \\
     \cline{2-9}
     & $H^1$-norm           & 3.08e-1 & 1.91e-1 & 1.26e-1 & 8.44e-2 & 5.83e-2 & $\approx 1.46$ &  $O(h^\frac12)$\\
     \hline
     \end{tabular}
\end{center}
\end{table}
\begin{table}[!ht]
\caption{Error for lumped mass FEM: initial data Dirac $\delta$-function, $\al=0.5$}
\label{tab:Deltafun}
\begin{center}
     \begin{tabular}{|c|c|c|c|c|c|c|c|c|}
     \hline
     Time& $h$& $1/8$ & $1/16$ &$1/32$ &$1/64$ & $1/128$ & ratio & rate \\
   \hline
     $t=0.001$ & $L_2$-norm  & 1.98e-1 & 7.95e-2 & 3.00e-2 & 1.09e-2 & 3.95e-3 & $\approx 2.75$&  $O(h^\frac32)$ \\
     \cline{2-9}
     & $H^1$-norm           & 5.56e0 & 4.06e0 & 2.83e0 & 2.02e0 & 1.41e0 &$\approx 1.42$ &  $O(h^\frac12)$\\
    \hline
    $t=0.01$ & $L_2$-norm  & 6.61e-2 & 2.56e-2& 9.51e-3 & 3.47e-3  & 1.25e-3 & $\approx 2.78$ &  $O(h^\frac32)$\\
    \cline{2-9}
    & $H^1$-norm           & 1.84e0 & 1.30e0 & 9.10e-1 & 6.40e-1 & 4.47e-1 & $\approx 1.42$ &  $O(h^\frac12)$\\
    \hline
      $t=0.1$ & $L_2$-norm  & 2.15e-2 & 8.13e-3 & 3.01e-3 & 1.09e-3 & 3.95e-4& $\approx 2.75$ &  $O(h^\frac32)$\\
     \cline{2-9}
     & $H^1$-norm  & 5.87e-1 & 4.14e-1 & 2.88e-1 & 2.03e-1 & 1.41e-1 & $\approx 1.43$ &  $O(h^\frac12)$\\
     \hline
     \end{tabular}
\end{center}
\end{table}
\begin{figure}[ht!]
\center
\includegraphics[width=11cm]{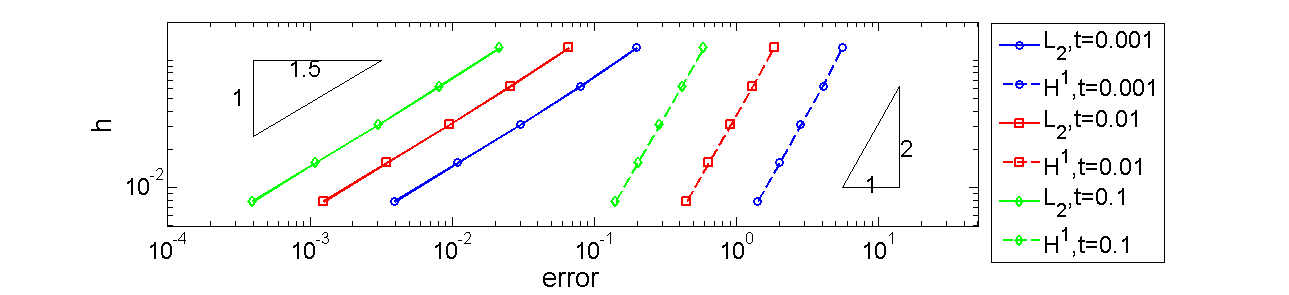}
\caption{Error plots for Example (d): initial data Dirac $\delta$-function, $\al=0.5$.}
\label{fig:weak2D}
\end{figure}

\bibliographystyle{abbrv}
\bibliography{frac}

\end{document}